\documentclass{amsart}
\usepackage{bbm}
\usepackage[dvips]{graphicx}
\usepackage[latin1]{inputenc}

\newtheorem{prop}{Proposition}
\newtheorem{theorem}{Theorem}
\newtheorem{lemma}{Lemma}

\newtheorem{defi}{Definition}
\newtheorem{corollary}{Corollary}

\def\N{{\mathbb N}}
\def\R{{\mathbb R}}

\def\P{{\mathbb P}}
\def\E{{\mathbb E}}

\def\eps{\varepsilon}

\newcommand{\diff}{\mathop{}\mathopen{}\mathrm{d}}
\newcommand\ind[1]{\mathbbm{1}_{\left\{#1\right\}}}
\newcommand\croc[1]{\left\langle #1\right\rangle}

\def\cal{\mathcal}
\def\eps{\varepsilon}

\title[Scaling Analysis of a Transient Stochastic Network]{A Scaling Analysis of a Transient Stochastic Network (I)}

\author{Mathieu Feuillet}
\address[M. Feuillet,Ph. Robert]{INRIA Paris---Rocquencourt, Domaine de Voluceau, 78153 Le Chesnay, France}
\email{Mathieu.Feuillet@inria.fr}
\author{Philippe  Robert}
\email{Philippe.Robert@inria.fr}
\urladdr{http://www-rocq.inria.fr/\string~robert}
\date{\today}
\keywords{Time Scales; Stochastic Averaging Principle; Transient Markov Chains with Absorbing State; Skorokhod Problem}

\begin{document}

\begin{abstract}
In this paper, a simple transient Markov process with an absorbing point is used to investigate the qualitative behavior of a large scale storage network of non reliable file servers where files can be duplicated. When the size of the system goes to infinity it is shown that there is a critical value for the maximum number of files per server such that below this quantity, the system stays away from the absorbing state, all files lost, in a quasi-stationary state where most files have a maximum number of copies. Above this value, the network looses a significant number of files until some equilibrium is reached. When the network is stable, it is shown that, with  convenient time scales, the evolution of the network towards the absorbing state can be described via a stochastic averaging principle. 
\end{abstract}

\maketitle

\bigskip

\hrule

\vspace{-3mm}

\tableofcontents

\vspace{-1cm}

\hrule

\bigskip

\section{Introduction}
\subsection*{Storage systems}
One considers a large scale storage system, it is a set of file servers in a communication network.  In order to ensure persistence, files are duplicated on several servers. When the disk of a given server breaks down, its files are lost but  can be retrieved on the other servers if copies are available. For these architectures a fraction of the bandwidth  of a server is devoted to the duplication mechanism of its files to other servers. On one hand, there should be sufficiently many copies so that any file has a copy available on at least one server  at any time. On the other hand, in order to use  the bandwidth in an optimal way, there should not be too many copies of a given file so that the network can accommodate a large number of distinct files. These systems are known as distributed hash tables (DHTs),  they play an important role in the development of some large scale distributed systems. see Rhea et al.~\cite{rhea-05} and Rowstron and Druschel~\cite{rowstron-01} for a more detailed presentation. 

Failures of disks occur naturally randomly, these events are quite rare but, given the large number of nodes of these distributed systems, this is not a negligible phenomenon at the level of the network. If, for a short period of time, several of the servers break down, it may happen that files will be lost for good just because all the available copies were on these servers and because the recovery procedure was not completed before the last copy disappeared. To design such a system, it is therefore desirable to find a convenient duplication policy and to dimension the system so that all files will have at least a copy as long as possible. The natural critical parameters of the network are the failure rates of servers, the bandwidth allocated to duplication, the number of files and the number of servers. The ratio of the two last quantities being a measure of the storage capacity of the system. It is important to understand the impact of each of these parameters on the efficiency of the storage
system. 

\subsection*{Stochastic Models}
This network can be seen as a classical set of queues with breakdowns. Numerous stochastic models of such systems have been investigated in the literature, see Chapter~6 of King~\cite{King} for example and the references therein. Related models concern queues with retrial and queues with servers of walking types, see Artalejo and G{\'o}mez-Corral~\cite{Arta} and Falin and Templeton~\cite{Falin}. For most of the systems analyzed, there are, in general, one or two nodes which are subject to breakdowns.  A queueing analysis is generally done in this context: convergence in distribution of the associated Markov model and analysis of the distribution of the availability of the system, of the delays and of queue sizes, \ldots  For DHTs, the rare stochastic models to investigate their performances describe the evolution of the number of copies of a given file. See Chun et al.~\cite{chun-06}, Picconi et al.~\cite{picconi-07} and Ramabhadran and Pasquale~\cite{ramabhadran-06}. See also Feuillet and Robert~\cite{Feuillet}. In most of these studies  the interaction between different files, due to the bandwidth sharing limitations, has not been really considered, except through simulations. The purpose of this paper is to investigate the impact of this interaction. The second important aspect is that  a large system, i.e. with a large number of files, will be considered instead of a small number of elements. This assumption is quite natural for current distributed systems.   

More precisely, the following simple model is considered:  A file can have at most two copies, the total bandwidth allocated to file duplication is given by $\lambda N$, for $\lambda>0$ and $N\in\N$. If at some moment there are $ x\geq 1$ files with exactly one copy, a new copy of each of these files is created at rate $\lambda N/x$.  It is assumed that initially $F_N$ files are present in the system with two copies and each copy of a file disappears at rate $\mu$. Recall that a file with $0$ copies is lost.  It will be assumed that the total number of files $F_N$ is proportional to $N$, i.e. that $F_N/N$ converges to some $\beta>0$. Clearly enough, this system is transient and the empty state, all files are lost,  is an absorbing state. The aim of this paper is of describing the decay of the network, i.e. how the set of lost files in increasing.  For $\delta>0$, there exists some finite random instant $T_N(\delta)$, such that a fraction $\lfloor \delta N\rfloor$ of the files are lost after time  $T_N(\delta)$. The paper investigates the order of magnitude in $N$ of the variables $T_N(\delta)$ as $N$ gets large and the role of the parameters $\lambda$, $\mu$
and $\beta$ in these asymptotics. 

In practice, if there are $N$ servers and that each of them has an available bandwidth $\lambda$ to duplicate files, the maximal capacity for duplication is then $\lambda N$. The model described above has therefore an optimal use of the duplication mechanism since the maximal duplication capacity is always available. For this reason this model provides upper bounds on the optimal performances of such  a system.  In particular, for any duplication mechanisms, after a duration of time with the same distribution as $T_N(\delta)$, at least $\lfloor \delta N\rfloor$ files will be lost for good.  A more realistic model, when the total duplication bandwidth is not anymore centralized, is investigated in Feuillet and Robert~\cite{Feuillet:06} via  mean-field limit asymptotics. It turns out that the corresponding mean-field limit can in fact be expressed in terms of the simple model analyzed in this paper.  The more general case when there are at most $d\geq 2$ copies of a given file will be investigated in another  sequel to this paper.

\subsection*{Time Scales of Transient Markov Processes}
If, for $i\in\{0,1,2\}$, $X_i^N(t)$ denotes the number of files with $i$ copies in the network, then, under Poisson assumptions for failures and for duplication processes, $(X_0^N(t),X_1^N(t))$ is clearly a finite Markov process with $(F_N,0)$ as an absorbing state. At the difference of previous works mentioned above, there is clearly no question of equilibrium here since the system {\em dies} at $(F_N,0)$. A possible approach to investigate the  decay of such a  system could be of considering the associated quasi-stationary distributions of the Markov process. See Darroch and Seneta~\cite{Darroch} and Ferrari et al.~\cite{Ferrari} for example. It would give a description of the system conditionally on the event that only a fraction  of the files has been lost. These quantities are generally expressed in terms of the spectral characteristics of the jump matrix. For this reason, explicit description of these distributions are quite rare outside one dimensional birth and death processes. In this paper, different time scales will be used to investigate the qualitative behavior of these transient processes. Times scales can be thought as ``lenses'', two of them that will focus on the stable part of the sample path of the  process (if any), this will give at the same time a kind of associated quasi-stationary distribution. Finally, a third time scale will focus on the decaying part of the sample paths, i.e. when the proportion of lost files is steadily increasing. 
\subsection*{Stochastic Averaging Principles} 
It is shown that in some cases,  a stochastic averaging principle (SAP) occurs for this transient process: roughly speaking its dynamics can be decomposed into two components, one evolving on a fast time scale and the other one on a slower time scale. The system is fully coupled in the sense that the jump rates of the slow process  depends on the equilibrium of the fast process, and the jump rates of the fast process depends of the state of the slow process. See Khasminskii~\cite{Khasminski:01} and Freidlin and Wentzell~\cite{Freidlin:02}. This phenomenon is known to occur for the classical example of loss networks. In this case the vector of the number of free places of the congested links is the fast component, see Kelly~\cite{Kelly:04} and Hunt and Kurtz~\cite{Hunt:05}. Outside this class of networks, there are, up to now, few examples of  stochastic networks for which a  fully coupled  SAP occurs. See Feuillet~\cite{Feuillet12} and Perry and Whitt~\cite{Whitt} for recent examples of SAP. 

This SAP phenomenon is already well known in the framework of deterministic dynamical systems, see Guckenheimer and Holmes~\cite{Guckenheimer}. In a stochastic context, an additional difficulty, sometimes underestimated, is of controlling the regularity properties of the family of invariant distributions indexed by the states of the slow process, instead of the family of fixed points in the deterministic case. This can be done through a kind of uniform control of some ergodic averages, see Freidlin and Wentzell~\cite{Freidlin:02} or by using a martingale representation of the associated Markov processes, see Kurtz~\cite{Kurtz:05}. In any case, there are several delicate technical issues to address: a convenient tightness result for a set of random measures and the rate of convergence of ergodic averages. In this paper,  a martingale formulation is also used but with a technical background significantly reduced. By taking a convenient state space for random measures, technical  results related to extensions of  random measures with specific measurability properties are not necessary. Furthermore,  the tightness of the family of invariant distributions of fast processes is obtained as a consequence of a simple monotonicity property. If the monotonicity property is quite specific, it seems that the method to avoid extension results can be used in a quite general framework. This will be the subject of further investigations.   

\subsection*{Outline of the Paper}
Section~\ref{modelsec} introduces the Markov process investigated and its corresponding martingale representation. Section~\ref{FluidSec} studies a fluid picture of the network, i.e. the limit of the sequence of processes $(X_0^N(t)/N,X_1^N(t)/N)$, it is shown in Theorem~\ref{fluidtheo} that its limit, the solution of an ODE, is not trivial when $\lambda <2\mu\beta$ and is $(0,0)$ when $\lambda > 2\mu\beta$. The storage system is therefore properly designed when $\lambda >2\mu\beta$, otherwise it is inefficient since it is losing a significant number of files right from the beginning. Section~\ref{heavysec} is devoted to the critical case $\lambda=2\mu\beta$, Theorem~\ref{Heavy} shows that the sequence of processes $(X_0^N(t)/\sqrt{N},X_1^N(t)/\sqrt{N})$ is converging in distribution and that its limit can be expressed in terms of a non-Markovian one-dimensional process, solution of an unusual stochastic differential equation with reflection at $0$. In Section~\ref{AverageSec}, the stable case $\lambda >2\mu\beta$ is investigated. It is shown that the capacity of the system remains intact at the normal time scale: For $t\geq 0$,   Theorem~\ref{NormalProp} proves that the variable $(X_0^N(t))$ converges in distribution to a  Poisson process. Only a finite number of files is lost as $N$ goes to infinity. More interesting, Theorem~\ref{theodec} shows that on the time scale $t\to Nt$ the transience of the Markov process  shows up: at ``time'' $Nt$ a fraction $\psi(t)N$ of the files is lost where $\psi(t)$ is the solution of some fixed point equation. This is the case where a stochastic averaging principle holds: around time $Nt$ there is a local equilibrium for which  $(\beta-\Psi(t))N$  files are still available. As a consequence,  $t\to Nt$ is the convenient  time scale to observe the degradation of the storage system. The proof of the convergence results use a more or less straightforward extension of the classical Skorokhod problem formulation, see Skorokhod~\cite{Skorokhod}. The necessary material is gathered in the appendix to keep the paper self-contained. 
\section{The Stochastic Model}\label{modelsec}
Recall that $F_N$ is the total number of distinct files initially present in the network and $X_1^N(t)$, resp. $X_0^N(t)$ is the number of files with one copy at time $t$, the number of lost files at this instant. The number $X_2^N(t)$ of files with two copies at time $t$ is defined by  $X_2^N(t)=F_N-X_0(t)-X_1(t)$. In general it will be assumed that all files have the maximum number of copies initially. The copy of a file is lost with rate $\mu$ and, conditionally on $X_1^N(t)=x$, a file with only one copy gets an additional copy with rate $\lambda N/x$. All events are supposed to occur after an exponentially distributed amount of time. Under these assumptions $(X(t))\stackrel{\text{def.}}{=}(X_0^N(t),X_1^N(t))$ is a Markov process on the state space 
\[
{\cal S}=\{x=(x_0,x_1)\in\N^2: x_0+x_1\leq F_N\},
\]
as mentioned above, with these assumptions, the state $(F_N,0)$ is an absorbing point of the process $(X^N(t))$. 

For $x\in\N^2$, the $Q$-matrix $Q^N=(q^N(\cdot,\cdot))$ of the process $(X(t)))$ is defined by 
\begin{equation}\label{Qmat}
\begin{cases}
q^N(x,x+e_1)=2\mu (F_N-x_0-x_1),\\
q^N(x,x-e_1)=\lambda N\ind{x_1>0},\\
q^N(x,x-e_1+e_0)=\mu x_1.
\end{cases}
\end{equation}
It is assumed that 
\begin{equation}\label{beta}
\lim_{N\to+\infty}{F_N}/{N}=\beta,
\end{equation}
and one denotes $\rho=\lambda/\mu$.  

The stochastic differential equations associated to this transient Markov process can be written as 
\begin{align}
X^N_0(t)&=X^N_0(0)+\sum_{i=1}^{+\infty} \int_0^t\ind{i\leq X_1^N(u-)} {\cal N}_{\mu,i}(\diff u),\label{SDE01}\\
X^N_1(t)&= X^N_1(0) -\int_0^t\ind{X_1^N(u-)>0} {\cal N}_{\lambda N}(\diff u) -\sum_{i=1}^{+\infty} \int_0^t\ind{i\leq X_1^N(u-)} {\cal N}_{\mu,i}(\diff u)\label{SDE02}
\\ &\quad +\sum_{i=1}^{+\infty} \int_0^t\ind{i\leq F_N-X_0^N(u-)-X_1^N(u-)} {\cal N}_{2\mu,i}(\diff u),\notag
\end{align}
where $({\cal N}_{\mu,i})$ and  $({\cal N}_{2\mu,i})$ are two i.i.d. independent sequence of Poisson processes with respective parameters $\mu$ and $2\mu$, ${\cal N}_{\lambda N}$ is an independent Poisson process with parameter $\lambda N$.  For the $i$th  file  having only one copy, the integrand of the right hand side of Relation~\eqref{SDE01} corresponds to its definitive loss  and the first term of the right hand side of Relation~\eqref{SDE02} is associated to its duplication. The last term of Relation~\eqref{SDE02} represents the loss of a copy of files with two copies. 

Relation~\eqref{SDE01} can be rewritten as 
\begin{equation}\label{SDE1}
X^N_0(t)=X^N_0(0)+\mu\int_0^t X^N_1(u)\,\diff u+M_0^N(t),
\end{equation}
where $(M_0^N(t))$ is the martingale defined by
\[
M_0^N(t)=\sum_{i=1}^{+\infty} \int_0^t\ind{i\leq X_1^N(u-)} \left[{\cal N}_{\mu,i}(\diff u)-\mu \diff u\right],
\]
its increasing process is given by
\[
\croc{M_0^N(t)}=\mu\int_0^t X_1^N(u)\,\diff{u},
\]
in particular, since $X_1^N(u)\leq F_N$, there exists some constant $C_0$ such that 
\[
\E\left(M_{0}^N(t)^2\right)=E\left(\croc{M_0^N(t)}\right)\leq C_0 N t
\]
holds for all $t\geq 0$ and $N\geq 1$.

Similarly, if  $f$ is in $C_c(\N)$, the set of functions  with finite support on $\N$, Relation~\eqref{SDE02} gives the representation
\begin{multline}
f(X^N_1(t))=f(X^N_1(0))+\mu\int_0^t\left[f(X_1^N(u)-1)-f(X_1^N(u))\right]X_1^N(u)\,\diff u\label{SDE2}\\+
N\int_0^t \Omega\left[\frac{F_N}{N}-\frac{X^N_1(u)+X^N_0(u)}{N}\right](f)(X(u))\,\diff u+M_{1}^N(t),
\end{multline}
where , for $y\geq 0$, $\Omega[y]$ is the functional operator defined by 
\begin{equation}\label{Omega}
\Omega[y](f)(x)= 2\mu y(f(x+1)-f(x))+ \lambda \ind{x>0}(f(x-1)-f(x)),\quad x\in\N,
\end{equation}
and $(M_{1}^N(t))$ is a martingale such that, for some constant $C_1$,
\[
\E\left(M_{1}^N(t)^2\right)\leq C_1N \|f\|_{\infty,F_N} t,
\]
holds for all $t\geq 0$ and $N\geq 1$, where $\|f\|_{\infty,F_N}=\max\{|f(x)|:0\leq x\leq F_n\}$.

\section{The Overloaded Network}\label{FluidSec}
In this section, it is proved that a significant fraction of files is lost quickly if the network is not correctly dimensioned, i.e.  when the ratio $\rho=\lambda/\mu$ is less than $2\beta$.  In this case, for a large $N$, the fraction of files with two copies at time $t$, $(F_N-X_0^N(t)-X_1^N(t))/N$ is close to $\rho/2$ if $t$ is large enough. As a consequence $(\beta-\rho/2)N$ files are lost and the network stabilizes with a subset of files with two copies whose cardinality is of the order of $\rho/2$. This is the critical case which is analyzed in Section~\ref{heavysec} where it is proved that the number of files lost files is of the order of $\sqrt{N}$. When $\rho>2\beta$, no file is lost at the fluid level. This case is investigated precisely in Section~\ref{AverageSec}.  
\begin{theorem}[Fluid Equations]\label{fluidtheo}
If  $(X_0^N(0),X_1^N(0))$ is some fixed element of ${\cal S}$ and
$\lim_{N\to+\infty}{F_N}/{N}=\beta $
then the sequence of processes $({X_0^N(t)}/{N}, {X_1^N(t)}/{N})$ converges in distribution to 
\[
\begin{cases}
\left[(\beta-\rho/2)(1- 2e^{-\mu t}+e^{-2\mu t}),(2\beta-\rho)\left(e^{-\mu t}-e^{-2\mu t}\right)\right]
& \text{ if } \rho\leq 2\beta,\\
(0,0)& \text{ if } \rho > 2\beta. 
\end{cases}
\]
\end{theorem}
\begin{proof}
Let $(X^N_0(0),X^N_1(0))=(y_0,y_1)\in {\cal S}$. 
Equations~\eqref{SDE1} and~\eqref{SDE2}, with the function $f\equiv\rm{Id}$ on $[0,F_N]$,  can be written as 
\begin{align}
\frac{X^N_0(t)}{N}&=\frac{y_0}{N}+\mu\int_0^t \frac{X^N_1(u)}{N}\,\diff u+\frac{M_0^N(t)}{N},\label{eq11}\\
\label{eq2}
\frac{X^N_1(t)}{N}&=\frac{y_1}{N}+2\mu\int_0^t \left(\frac{F_N}{N}-\frac{X^N_1(u)+X^N_0(u)}{N}\right)\,\diff u-\lambda t\\
&-\mu\int_0^t\frac{X_1^N(u)}{N}\,\diff u +\frac{M_{1}^N(t)}{N}+\lambda\int_0^t\ind{X_1^N(u)=0}\,\diff u.\notag
\end{align}
Doob's Inequality and the bounds on the second moments of the associated martingales show that, for $i=0$, $1$ and $t\geq 0$, 
\[
\P\left(\sup_{0\leq s\leq t}\frac{M_i^N(s)}{N}\geq \eps\right)\leq \frac{1}{\eps^2}\E(M_i(t)^2)\leq \frac{1}{N}\frac{C_i t}{\eps^2}.
\]
 Therefore, the two sequences of processes $(M_0^N(t)/N)$ and
$(M_1^N(t)/N)$ converge in distribution to $0$ uniformly on compact sets. 

For $T>0$, $\delta>0$ and   for $i=0$, $1$, define  $w^T_{X_i^N}(\delta)$ as the modulus of 
continuity of the process $(X_i^N(t))$ on the interval $[0,T]$,
\begin{equation}\label{modcont}
w^T_{X_i^N}(\delta)=\sup_{0\leq s\leq t\leq T,\,|t-s|\leq \delta} \left|X_i^N(t)-X_i^N(s)\right|. 
\end{equation}
By using the fact that,  for some constant $C$, $X_i^N(t)\leq F_N\leq CN$ for all $N\in\N$
and $t\geq 0$, the above equations and the convergence of the martingales to $0$ give that, 
for any $\eps>0$ and $\eta>0$, there exists $\delta>0$ such that the relation
$\P(w^T_{X_i^N}(\delta)\geq \eta)\leq \eps$ holds for all $N$.

This implies that the sequence of stochastic processes $(X_0^N(t)/N,X_1^N(t)/N)$ is
tight. See Billingsley~\cite{Billingsley} for example. 
One denotes by $(x_0(t),x_1(t))$ a limiting value for some subsequence
$(N_k)$. From Equation~\eqref{eq11}, one gets the relation
\[
x_0(t)=\mu\int_0^t x_1(u)\,\diff u.
\]
Define
\begin{equation}\label{eqz}
Z^N(t)= \mu\int_0^t \left(\frac{2F_N}{N}-\frac{3X^N_1(u)}{N}-\frac{2X^N_0(u)}{N}\right)\,\diff u-\lambda t +\frac{M_{1}^N(t)}{N},
\end{equation}
Equation~\eqref{eq2} can be also interpreted as the fact that 
\begin{equation}\label{xrn}
(X_Z^N(t),R_Z^N(t))\stackrel{\text{def.}}{=}\left(X_1^N(t)/N, \lambda \int_0^t\ind{X_1^N(u)=0}\,\diff u\right)
\end{equation}
is the unique solution of the Skorokhod problem associated to the process $(Z^N(t))$. See Appendix for a definition. 

The sequence  $(Z^{N_k}(t))$ is converging in  distribution and by the continuous mapping theorem
\begin{align}\label{aux1}
\lim_{k\to+\infty}(Z^{N_k}(t))=y(t)&\stackrel{\text{def.}}{=}
\mu\int_0^t \left(2\beta-2x_0(u)-3x_1(u)\right)\,\diff u-\lambda t\\
&=(2\mu\beta-\lambda)t-\mu\int_0^t\left(3 x_1(u)+2\mu\int_0^u x_1(v)\,\diff v\right)\,\diff u\notag
\end{align}
The solutions of Skorokhod problems being continuous with
respect to the process $(Z^N(t))$, see Appendix~D of Robert~\cite{Robert} for example, one
gets that  $(X_Z^N(t),R_Z^N(t))$ converges in distribution to the solution $(x_y(t),r_y(t))$ of the Skorokhod
problem associated to $(y(t))$. Since $x_y(t)=x_1(t)$ and $y(t)=F(x_1)(t)$ with
\[
F(x)(t)=(2\mu\beta-\lambda)t-\mu\int_0^t\left(3x(u)+2\mu\int_0^u x(v)\,\diff v\right)\,\diff u,
\]
the process $(x_1(t))$ is a solution of the generalized Skorokhod problem (GSP) associated to the functional $F$. See Appendix.  Proposition~\ref{GSPprop} shows that such a solution exists and is unique.  This implies that there is a unique, deterministic limiting value for the sequence $(X_0^N(t)/N,X_1^N(t)/N)$.  It is easy to check that the explicit expressions for $(x_0(t))$ and $(x_1(t))$ given in the statement of the theorem are indeed the solutions of the GSP. The convergence in distribution is therefore established.
\end{proof}

\section{The Critical case}\label{heavysec}
To complete  the picture of  the overloaded network  $\rho\leq 2\beta$, one  considers the
critical case $\rho=2\beta$.  As it will be seen, the convergence result is expressed  in
terms of a reflected stochastic differential equation. The  appendix  presents the
corresponding definition  and a  result of  existence and uniqueness.
\begin{theorem}\label{Heavy}
If $\lambda/\mu=2\beta$ and, for some $\gamma\in\R$, 
\[
\lim_{N\to+\infty} \frac{1}{\sqrt{N}}\left(F_N-N\frac{\rho}{2}\right)=\gamma
\text{ and } \lim_{N\to+\infty} \frac{X_1^N(0)}{\sqrt{N}}=y,
\]
and $X_0^N(0)=0$, then for the convergence in distribution
\[
\lim_{N\to+\infty}\left(\frac{X_0^N(t)}{\sqrt{N}},\frac{X_1^N(t)}{\sqrt{N}}\right)=\left(\mu\int_0^tY(u)\,\diff u,Y(t)\right), 
\]
where $(Y(t))$ is the solution starting at $y$ of  the  stochastic differential equation
\begin{equation}\label{rsde}
\diff Y(t)=\sqrt{2\lambda}\,\diff B(t) +\mu\left(2\gamma-3Y(t)-2\mu\int_0^t Y(u)\,\diff u\right)\,\diff t 
\end{equation}
reflected at $0$, i.e. with the constraint that $Y(t)\geq 0$, for all $t\geq 0$. The process $(B(t))$ is a standard Brownian motion on $\R$. 
\end{theorem}
The solution of SDE~\eqref{rsde}, is non-Markovian due to the integral term in the drift.
\begin{proof}
Equations~\eqref{SDE1}  and ~\eqref{SDE2}, with the function $f\equiv\rm{Id}$ on $[0,F_N]$,  can be written as 
\begin{align}
\overline{X}^N_0(t) &\stackrel{\text{def.}}{=}\frac{X^N_0(t)}{\sqrt{N}}=\mu\int_0^t \frac{X^N_1(u)}{\sqrt{N}}\,\diff u+\frac{M_0^N(t)}{\sqrt{N}},
\label{SDE11}\\
\overline{X}^N_1(t)&\stackrel{\text{def.}}{=}\frac{X^N_1(t)}{\sqrt{N}}=\frac{X^N_1(0)}{\sqrt{N}}+2\mu\int_0^t
\left(\gamma_N-\frac{X^N_1(u)}{\sqrt{N}}-\frac{X^N_0(u)}{\sqrt{N}}\right)\,\diff u\label{SDE22}\\&\qquad
-\mu\int_0^t \frac{X_1^N(u)}{\sqrt{N}}\,\diff u+\frac{M_{1}^N(t)}{\sqrt{N}}+ \lambda \sqrt{N}\int_0^t\ind{X_1^N(u)=0}\,\diff u,
\notag
\end{align}
with $\gamma_N=(F_N-N\rho/2)/\sqrt{N}$. 
With the same notations as in Section~\ref{modelsec}, the martingales $(M_0^N(t))$ and $(M_1^N(t))$ 
are 
\begin{align*}
M_0^N(t)&=\sum_{i=1}^{+\infty}\int_0^t \ind{i\leq X_1^N(u-)}[{\cal N}_{\mu, i}(\diff u)-\mu\,\diff u]\\
M_1^N(t)&=\sum_{i=1}^{+\infty} \int_0^t \ind{i\leq F_N- X_1^N(u-)-X_0^N(u-)}[{\cal N}_{2\mu}(\diff u)-2\mu\,\diff u]\\
&-M_0^N(t)
-\int_0^t \ind{X_1^N(u)>0} [{\cal N}_{\lambda N}(\diff u)-\lambda N\,\diff u].
\end{align*}
Their increasing processes are given by
\begin{align*}
\croc{\frac{1}{\sqrt{N}}M_0^N}(t)&= \mu\int_0^t \frac{X_1^N(u)}{N}\,\diff u,\\
\croc{\frac{1}{\sqrt{N}}M_1^N}(t)&
= 2\mu\int_0^t \left(\frac{F_N}{N}- \frac{X_1^N(u)}{N}-\frac{X_0^N(0)}{N}\right)\,\diff u
+\croc{\frac{1}{\sqrt{N}}M_0^N}(t)\\ &\qquad 
+\lambda \int_0^t\ind{X_1^N(u)>0}\,\diff u.
\end{align*}
The  last term of  the right  hand side  of the  above equation  is $(R^N(t))$  defined by
Equation~\eqref{xrn} in the proof of the  previous theorem.  It is the second component of
the  solution   to  the  Skorokhod  problem   associated  to  the   process  $(Z^N(t))$  of
Relation~\eqref{eqz}.   It has  been seen  that the  sequence of  processes  $(Z^N(t))$ is
converging  to  $(y(t))$  defined  in  Equation~\eqref{aux1}. In  this  case  $(y(t))$  is
identically $0$, the solution of the corresponding Skorokhod problem associated to $(y(t))$
is therefore $(0,0)$.
The continuity properties of the solutions of the Skorokhod  problem imply that the
process $(R^N(t))$ converges to  $0$. Consequently, by Theorem~\ref{fluidtheo} one gets the
convergence in distribution 
\[
\lim_{N\to+\infty} \left(\int_0^t\ind{X_1^N(u)=0}\,\diff u\right)=0
\]
and therefore
\[
\lim_{N\to+\infty} \left(\croc{\frac{1}{\sqrt{N}}M_0^N}(t)\right)=0 \text{ and }
\lim_{N\to+\infty}\left(\croc{\frac{1}{\sqrt{N}}M_1^N}(t)\right) = (2\lambda t).
\]
One deduces that $(\overline{M}_1^N(t))\stackrel{\text{def.}}{=}(M_1^N(t)/\sqrt{N})$ converges to $(\sqrt{2\lambda}B(t)))$ where $(B(t))$ is a standard Brownian motion and that $(\overline{M}_0^N(t))\stackrel{\text{def.}}{=}(M_0^N/\sqrt{N})$ converges to $0$.  See Ethier and Kurtz~\cite{Ethier}  for example. 

One now proves that the processes 
\[
\left(\overline{X}_0^N(t)\right)\stackrel{\text{def.}}{=}\left(\frac{X^N_0(t)}{\sqrt{N}}\right) \text{ and } 
\left(\overline{X}_1^N(t)\right)\stackrel{\text{def.}}{=}\left(\frac{X^N_1(t)}{\sqrt{N}}\right)
\]
are tight. If $(h(t))$ is a function $\R_+$, one
denotes, 
\[
\|h\|_{\infty,t}=\sup_{0\leq s\leq t} |h(s)|
\]
and $w_h^t(\cdot)$ is the modulus of continuity of $h$ defined by Equation~\eqref{modcont}. 
Equation~\eqref{SDE11} gives, for $0\leq t\leq T$,
\[
\left\|\overline{X}_0^N\right\|_{\infty,t}\leq  \left\|\overline{M}_0^N\right\|_{\infty,T}+\mu\int_0^t \left\|\overline{X}_1^N\right\|_{\infty,u}\,\diff u.
\]
Equation~\eqref{SDE22} shows that $(X_1^N(t)/\sqrt{N})$ is the
first coordinate of the solution of the Skorokhod problem associated to $(Z_1^N(t))$
defined by 
\begin{equation}\label{eqz1}
Z_1^N(t)\stackrel{\text{def.}}{=} y_N+
\mu\int_0^t \left(2\gamma_N-3\frac{X^N_1(u)}{\sqrt{N}}+2\frac{X^N_0(u)}{\sqrt{N}}\right)\,\diff u +\frac{M_{1}^N(t)}{\sqrt{N}},
\end{equation}
with $y_N=X^N_1(0)/\sqrt{N}$. By using the explicit representation of the solution of a Skorokhod problem in dimension~1, one has 
\[
\|\overline{X}_1^N\|_{\infty,t}\leq 2\|Z_1^N\|_{\infty,t}, \text{ for }0\leq t\leq T,
\]
see Appendix~D of Robert~\cite{Robert} for example, then 
\begin{align*}
\left\|\overline{X}_1^N\right\|_{\infty,t}&\leq 2 y_N+ 4\mu\gamma_N T +2\left\|\overline{M}_1^N\right\|_{\infty,T}
+4\mu \int_0^t \left(\left\|\overline{X}_0^N\right\|_{\infty,u}+\left\|\overline{X}_1^N\right\|_{\infty,u}\right)\,\diff u\\
& \leq U^N(T) + 
(4+\mu T)\mu\int_0^t \left\|\overline{X}_1^N\right\|_{\infty,u}\,\diff u,
\end{align*}
with
$U^N(T) \stackrel{\text{def.}}{=} 2y_N+4\mu\gamma_N T +2\|\overline{M}_1^N\|_{\infty,T}
+4\mu T\|\overline{M}_0^N\|_{\infty,T}$.

Gronwall's Inequality gives that the relation 
$ \|\overline{X}_1^N\|_{\infty,t}\leq U^N(T) \exp((4+\mu T)\mu t)$, 
holds for $0\leq t\leq T$, and, consequently, 
$$\left\|\overline{X}_0^N\right\|_{\infty,t}\leq \mu T U^N(T) e^{(4+\mu T)\mu t}+ 
\left\|\overline{M}_0^N\right\|_{\infty,T}.$$
The convergence of martingales shows that the two sequences of random variables $(U^N(T))$
and  $\|\overline{M}_0^N\|_{\infty,T}$ converge in  distribution. Consequently, for $\eps>0$, there exists some
$K>0$ such that for $i=0$, $1$ and all $N\geq 0$,
\[
\P\left(\|\overline{X}_i^N\|_{\infty,t}>K\right)\leq \eps.
\]
If $\eta>0$, there exists $N_0$ and $\delta$ sufficiently small so that, for all $N\geq N_0$, 
\[
2\mu\delta T(\gamma_N+2K)<\eta/2 \text{ and } \P\left(w_{\overline{M}_1^N}^T\geq \eta\right) \leq \eps.
\]
The last relation coming from the fact that the sequence $(\overline{M}_1^N(t))$  is
converging in distribution to a continuous process.  One gets  finally
\begin{align*}
\P&\left(w_{Z_1^N}^T(\delta) \geq \eta\right)\leq 
\P\left(2\mu\delta T\left[\gamma_N+\left\|\overline{X}_0^N\right\|_{\infty,T}+\left\|\overline{X}_1^N\right\|_{\infty,T}\right]+w_{\overline{M}_1^N}^T(\delta)\geq \eta\right)\\
&\leq \P\left(\left\|\overline{X}_0^N\right\|_{\infty,T}\geq K \right)+\P\left(\left\|\overline{X}_1^N\right\|_{\infty,T}\geq K \right)+
\P\left(w_{\overline{M}_1^N}^T(\delta)\geq \eta/2\right)\leq 3\eps.
\end{align*}
The sequence $(Z_1^N(t))$ is therefore tight, by continuity of the solution of the
Skorokhod problem the same property holds for $(X_1^N(t)/\sqrt{N})$ and consequently for
$(X_0^N(t)/\sqrt{N})$. 

If $(Y_0(t),Y_1(t))$ is a limit of a subsequence
$[(X_0^{N_k}(t)/\sqrt{N_k},X_1^{N_k}(t)/\sqrt{N_k})]$. By Equation~\eqref{SDE11} and~\eqref{SDE22}, one gets that
\[
\left(\frac{X_1^{N_k}(t)}{\sqrt{N_k}},\lambda \sqrt{N}\int_0^t\ind{X_1^N(u)=0}\,\diff u\right)
\]
converges in distribution to the solution of the Skorokhod problem associated to the process
\[
\left(y+\sqrt{2\lambda}B(t) + \mu\int_0^t \left(2\gamma-3Y_1(u)-2\mu\int_0^u Y_1(v)\,\diff v\right)\,\diff u\right).
\]
One concludes that $(Y_1(t))$ is the solution of the generalized Skorokhod problem for the
functional $F$ defined by 
\[
F(h)(t)=y+\sqrt{2\lambda}B(t) + \mu\int_0^t \left(2\gamma-3h(u)-2\mu\int_0^u h(v)\,\diff v\right)\,\diff u.
\]
Proposition~\ref{GSPprop} in the appendix shows that there is a unique solution
$(Y_1(t))$ and consequently a unique limit $(Y_0(t),Y_1(t))$. The theorem is proved.
\end{proof}
\section{The Time Scales of the Stable Network}\label{AverageSec}
The asymptotic properties of the network are investigated under the condition $\rho=\lambda/\mu>2\beta$. In Section~\ref{FluidSec} it has been shown that, in this case, the system is stable at the fluid level, i.e. that the fraction of lost files is $0$. Of course this does not change the fact that the system is still transient with the absorbing state $(F_N,0)$. To have a precise idea on how the system reaches this state, there are three interesting time scales to consider:
\begin{enumerate}
\item Slow time scale: $t\to t/N$,
\item Normal time scale: $t\to t$,
\item Linear time scale: $t\to Nt$,
\end{enumerate}
they are investigated successively in this section. 
The following  elementary lemma will be used throughout the section. 
\begin{lemma}\label{domlem}
If   $\rho=\lambda/\mu>2\beta$,  for any $\beta_0>\beta$ such that $\lambda/\mu >2\beta_0$, $\eps>0$, $\eta>0$ and $T>0$, there exists $N_0\in\N$ such that
\begin{enumerate}
\item Coupling: there exists a probability space where the relation  
\[
X^N_1(t)\leq L_{\beta_0}(N t), \forall t\geq 0,
\]
holds  for all $N\geq N_0$  and $t\geq 0$, with $(L_{\beta_0}(t))$  the process of the number of customers of an ergodic $M/M/1$ queue with arrival rate $2\mu \beta_0$ and service rate $\lambda$ and with initial condition $L_{\beta_0}(0)=X^N_1(0)$. 
\item The relation 
\[
\P\left[\sup_{\substack{0\leq s, t \leq T, \\|t-s|\leq \delta}} \frac{1}{N}\int_{sN}^{tN} L_{\beta_0}(u)\,\diff u >\eta\right]\leq \eps
\]
holds.
\end{enumerate}
\end{lemma}
\begin{proof}
There exists some $\beta_0\geq \beta$ and $N_0\geq 1$ such that $\lambda>2\beta_0\mu$ and that
$F_N\leq N\beta_0$ for $N\geq N_0$. It is enough to take the $M/M/1$ with arrival
rate $2\mu\beta_0$ and service rate $\lambda$. 

Denote by ${\cal A}$ the event on the left hand side of the  last relation to prove. If, for $x\in \N$,  $\tau_x$ denotes the hitting time of  $x$ by the process $(L_{\beta_0}(t))$, for $\delta<1/2$,  one has  
\[
\P({\cal A})\leq \P\left(\tau_{\lfloor \eta N\rfloor}\leq NT\right).
\]
By ergodicity of this process and  Proposition~5.11 of Robert~\cite{Robert} for example, there exists some $0<\alpha<1$ such that the sequence $(\alpha^N \tau_{\lfloor N\eta\rfloor})$ converges in distribution. The last term of the above relation is thus  arbitrarily small as $N$ gets large. 
\end{proof}

\subsection*{The slow time scale}
A description of the asymptotic behavior for the slow time scale is
presented informally. From Relation~\eqref{Qmat}, one can see that the $Q$-matrix of the process on the slow
time scale $(X_0^N(t/N),X_1^N(t/N))$ has the following asymptotic expansion
\[
\lim_{N\to+\infty}
\begin{cases}
q^N(x,x+e_1)&= 2\mu \beta,\\
q^N(x,x-e_1)&= \lambda \ind{x_1>0},\\
q^N(x,x-e_1+e_0)&= 0.
\end{cases}
\]
With elementary arguments which are skipped one can easily get the following proposition. It states that, on the slow time scale, with probability $1$ no file is lost at all in the limit. 
\begin{prop}
The sequence of processes $(X_0^N(t/N),X_1^N(t/N))$ converges in distribution to the process $(0,L_\beta(t))$, where $(L_\beta(t))$ is the process of the number of jobs of
an $M/M/1$ queue with arrival rate $2\mu\beta$ and service rate $\lambda$. 
\end{prop}
\subsection*{The normal time scale}
It is shown that, on the normal time scale, the stability does not only hold on the fluid level: almost surely there is a finite number of losses in any finite time interval, more precisely losses occur as a Poisson process.  See Proposition~\ref{NormalProp}.  The capacity $\lambda N$ of the network is thus able to maintain an almost complete set of files.  The following proposition shows in particular that the number of definitive losses at time $t>0$ is finite with a Poisson distribution. 

\begin{theorem}\label{NormalProp} If   $\rho=\lambda/\mu>2\beta$,
\begin{itemize} 
\item  the sequence of processes $(X_0^N(t))$ 
converges in  distribution  to  a Poisson point process  on $\R_+$ with rate
$2\mu\beta/(\rho-2\beta)$. 
\item For $t>0$, as $N$ goes to infinity, the random variable $X_1^N(t)$ converges in distribution to  a geometric distribution with parameter $2\beta/\rho$.  
\end{itemize}
\end{theorem}
The second convergence is for the marginal distribution of $(X_1^N(s))$ at time $t$. One cannot expect a convergence in distribution of the sequence of processes $(X_1^N(t))$. Indeed, since the sequence of processes $(X_1^N(t/N))$ is converging  in distribution to the law of the $M/M/1$ process $(L_\beta(t))$,  for $0\leq s< t$, the distribution of $(X_1^N(s),X_1^N(t))$ and of $(L_\beta(Ns),L_\beta(Nt))$ are close. Between time $Ns$ and $Nt$, the $M/M/1$ ``forgets'' its location at time $Ns$ (just because it hits $0$ with probability $1$)  so that when $N$ goes to infinity the couple $(X_1^N(s),X_1^N(t))$ converges in distribution to the distribution of two independent geometric distributions. The sample paths of a possible limit of $(X_1^N(s))$ would not have regularity properties. 
\begin{proof}
Define
\[
\eta_N(t)\stackrel{\text{def.}}{=} \int_0^tX_1^N(u)\,\diff u,
\]
for $0\leq s\leq t$, the above lemma gives that, 
\[
\eta_N(t)-\eta_N(s)=\int_s^tX_1^N(u)\,\diff u\leq \int_s^tL_{\beta_0}(Nu)\,\diff u=\frac{1}{N}\int_{Ns}^{Nt}L_{\beta_0}(u)\,\diff u.
\]
The criteria of the modulus of continuity and Lemma~\ref{domlem} give that the sequence of processes $(\eta^N(t))$ is tight. 
The above inequality and the ergodic theorem applied to the ergodic Markov process $(L_{\beta_0}(t))$ show also that, almost surely,
\begin{equation}\label{supeq}
\limsup_{N\to+\infty}\int_0^tX_1^N(u)\,\diff u\leq \frac{2\beta_0}{\rho-2\beta_0}t.
\end{equation}

For $T>0$ fixed and $K$ 
\begin{align*}
\P(X_0^N(T)\geq K)&\leq \P\left(\mu\int_0^T X^N_1(u)\geq K/2\right)+\P\left(M_0^N(T)\geq K/2\right)\\
&\leq \P\left(\mu\int_0^T X^N_1(u)\geq K/2\right)+\frac{4}{K^2}\E\left(\mu\int_0^T X_1^N(u)\,\diff{u}\right).
\end{align*}
One can thus choose $K$ so that $\P(X_0^N(T)\geq K)\leq \eps$ holds for $N\geq N_0$ for some $N_0\in\N$. 
As in the proof of Lemma~\ref{domlem}, for $\delta>0$, there exists some $N_1\in\N$ such that if $N\geq N_1$ then 
\[
\P\left(\sup_{0\leq s\leq T} L_{\beta_0}(s)\geq \delta N\right)\leq \eps. 
\]
In the same way as in the proof of the above lemma, one can construct an $M/M/1$ process $(Z^N(t))$
whose arrival and service rates are respectively
\[
2\mu\left(\frac{F_N}{N}-\frac{K}{N}-\delta\right) \text{ and } \lambda, 
\]
and such that, on the event,
\[
{\cal A}_T\stackrel{\text{def.}}{=}\left\{X_0^N(T)\leq K, \sup_{0\leq t\leq NT} L_{\beta_0}(t)\leq \delta N\right\},
\]
the relation $X_1^N(t)\geq Z^N(Nt)$ holds for all $t\leq T$. Hence, almost surely, 
\begin{equation}\label{infeq}
\liminf_{N\to+\infty}\eta^N(t)\geq \liminf_{N\to+\infty} \frac{1}{N} \int_0^{Nt} Z^N(u)\,\diff u=\frac{2(\beta-\delta)}{\rho-2(\beta-\delta)}t.
\end{equation}
holds on ${\cal A}_T$. By letting $\delta$ go to $0$ and $\beta_0$ to $\beta$ in Equations~\eqref{supeq} and~\eqref{infeq} respectively, one gets that the variable $\eta^N(t)$ converges almost surely to $\alpha t$ with $\alpha={2\beta\mu}/{(\rho-2\beta)}$. Consequently, the tightness of the sequence of processes $(\eta^N(t))$ implies that it is converging in distribution to $(\alpha t)$. 

Note that $t\mapsto X_0^N(t)$ can also be seen as a point process with jumps of size $1$. 
By Equation~\eqref{SDE01},  one has
\[
\left(X^N_0(t)-\mu\int_0^t X^N_1(u)\,\diff u\right)
\]
is a martingale with respect to the natural filtration of the associated Poisson processes. The random measure
\[
\Lambda^N([0,t])=\mu\int_0^t X^N_1(u)\,\diff u
\]
is a {\em compensator} of the point process $t\mapsto X_0^N(t)$. See Kasahara and Watanabe~\cite{Kasahara}. It has therefore been shown that the sequence of compensators is converging to the deterministic measure $\alpha\,\diff x$. Theorem~5.1 of~\cite{Kasahara}, see also Brown~\cite{Brown},  gives the convergence in distribution of $(X_0^N(t))$ to a Poisson process with rate $\alpha$. 

In a similar way as before, through the convergence of the $Q$-matrix, the asymptotic distribution of $X_1^N(t)$ can be easily obtained by conditioning on the event $\{X_0^N(t)\leq K\}$ for $K$ large and by using arbitrarily close $M/M/1$ processes at equilibrium as upper and lower stochastic bounds for $X_1^N(t)$.  Details are skipped. 
\end{proof}
\subsection*{The linear time scale $\mathbf{t\to Nt}$}
On the linear time scale, it will be shown that a fraction $\Psi(t)$ of the files is lost at time $t$. In some way the linear time scale gives a picture of the decay of the network.

For $N\geq 1$, the random measure $\mu_N$ on $\N\times\R_+$ is defined as, for a measurable function $g:\N\times\R_+\to\R_+$,
\[
\croc{\mu_N,g}=\int_{\R_+} g(X^N_1(Nt),t)\,\diff t.
\]
Note that if $g(x,t)=h(x)\ind{[0,T]}(t)$ for $T>0$, then
\[
\croc{\mu_N,g}=\sum_{x\in\N} h(x)\frac{1}{N}\int_0^{NT} \ind{X_1^N(t)=x}\,\diff t.
\]
Consequently $(\mu_N)$ is relatively compact sequence of random Radon measures on $\N\times\R_+$. See Dawson~\cite{Dawson} for example. Note that the measure identically null can be a possible limit of this sequence.

From now on, one fixes $(N_k)$ such that $(\mu_{N_k})$ is a converging subsequence whose limit is $\nu$. By taking a convenient probability space, one can assume that the convergence of $(\mu_{N_k})$ holds almost surely for the weak convergence of Radon measures.

Since, for $N\geq 1$, $\mu_{N}$ is absolutely continuous with respect to the product of the counting measure on $\N$ and Lebesgue measure on $\R_+$, the same property holds for the limiting measure $\nu$. Let $(x,t)\to \pi_t(x)$ denote its (random) density. It should be remarked that, one can choose a version of $\pi_t(x)$ such that the map $(\omega,x,t)\to \pi_t(x)(\omega)$ on the product of the probability space and $\N\times\R_+$ is measurable by taking $\pi_t(x)$ as a limit of measurable maps,
\[
\pi_t(x)=\limsup_{s\to 0}\frac{1}{s} \nu(\{x\}\times[t,t+s]).
\]
See Chapter~8 of Rudin~\cite{Rudin:01} for example.
\begin{prop}\label{proppsi}
For the convergence in distribution of continuous processes
\[
\lim_{k\to+\infty} \left(\frac{1}{N_k}\int_0^{N_kt}X_1^{N_k}(u)\,\diff u\right)
=(\Psi(t))\stackrel{\text{def.}}{=}\left(\mu \int_0^t \croc{\pi_u,I}\,\diff u\right),
\]
where $I(x)=x$ for $x\in\N$. Moreover, almost surely, for all $t\geq 0$,
\[
\int_0^t\pi_u(\N)\,\diff u =t.
\]
\end{prop}
It must be noted that the last relation is crucial, it shows that the masses of the measures $\mu_{N_k}$, for $k\geq 1$, do not vanish at infinity. This property is sometimes absent of the proofs of stochastic averaging principles, it is nevertheless mandatory to identify $\pi_u$ as an invariant distribution of a Markov process. 
\begin{proof}
 The criteria  of the modulus  of   continuity  is   used  to  prove   the  tightness  of   
\[
(\Psi_{N}(t))\stackrel{\text{def.}}{=} \left(\frac{\mu}{N}\int_0^{Nt}X_1^{N}(u)\,\diff u\right).
\]
By Lemma~\ref{domlem}
\[
\Psi_{N}(t)-\Psi_N(s)=\frac{\mu}{N}\int_{Ns}^{Nt}X_1^N(u)\,\diff u\leq \frac{\mu}{N^2}\int_{N^2s}^{N^2t}L_{\beta_0}(u)\,\diff u.
\]
As in the proof of Proposition~\ref{NormalProp}, one concludes that the  sequence of processes $(\Psi_N(t))$ is tight. 

For $K>0$ and $t\geq 0$, the almost sure convergence of the measures $(\mu_{N_k})$ gives the convergence 
\[
\lim_{k\to+\infty}\frac{1}{N_k}\int_0^{N_kt}X_1^{N_k}(u)\mathbbm{1}_{[0,K]}(X_1^{N_k}(u))\,\diff u= 
\int_0^t \croc{\pi_u,I\mathbbm{1}_{[0,K]}}\,\diff u,
\]
where $I(x)=x$. By using again Lemma~\ref{domlem}, one gets that
\[
\frac{1}{N_k}\int_0^{N_kt}X_1^{N_k}(u)\ind{X_1^{N_k}(u)\geq K}\,\diff u\leq 
\frac{1}{N_k^2}\int_0^{N_k^2t}L_{\beta_0}(u)\ind{L_{\beta_0}(u)\geq K}\,\diff u,
\]
and the ergodic theorem applied to $(L_{\beta_0}(t))$ shows that the last quantity is converging in distribution to 
\[
\left(t\,\E\left(L_{\beta_0}(\infty)\ind{L_{\beta_0}(\infty)\geq K}\right)\right)
\]
where $L_{\beta_0}(\infty)$ is the limit in distribution of $(L_{\beta_0}(t))$, a geometrically distributed random 
variable. For $\eps>0$, $K$ is chosen sufficiently large so that the last quantity is less than $\eps/2$, consequently if $k$ is large enough, one has
\[
\frac{1}{N_k}\int_0^{N_kt}X_1^{N_k}(u)\ind{X_1^{N_k}\geq K}\,\diff u\leq \eps.
\]
One deduces that $(\Psi(t))$ is the only possible limiting process for $(\Psi_{N_k}(t))$. This proves the first half of the proposition. 

For $K\geq 1$,  the convergence of $(\mu_{N_k})$ gives the relation
\[
\lim_{k\to+\infty}\frac{1}{N_k}\int_0^{N_kt}\ind{X_{N_k}(u)\leq K}\,\diff u=
\nu([0,K]\times [0,t])=
\int_0^{t}\pi_u([0,K])\,\diff u.
\]
By using again the stochastic domination by an ergodic $M/M/1$ queue, 
\[
\frac{1}{N_k}\int_0^{N_kt}\ind{L_{\beta_0}(u)\leq K}\,\diff u\leq \frac{1}{N_k}\int_0^{N_kt}\ind{X_{N_k}(u)\leq K}\,\diff u.
\]
by letting $k$ go to infinity one gets that, almost surely,
\[
t\P(L_{\beta_0}(\infty)\leq K)\leq \int_0^t\pi_u([0, K])\,\diff u\leq \int_0^t\pi_u(\N)\,\diff u,
\]
now if $K$ go to infinity, one obtains the relation
\[
\int_0^t\pi_s(\N)\,\diff s=t
\]
holds for all $t\in\N$ and consequently for all $t\geq 0$. The proposition is proved.
\end{proof}
\begin{theorem}[Rate of Decay of the Network]\label{theodec}
If $\rho=\lambda/\mu>2\beta$, then, as $N$ goes to infinity,  the process $(X_0^N(Nt)/N)$
converges to $(\Psi(t))$ where $\Psi(t)$ is the unique solution 
$y\in[0,\beta]$ of the equation 
\begin{equation}\label{decay}
\left(1-{y}/{\beta}\right)^{\rho/2} e^{y+\mu t}=1. 
\end{equation}
For $t\geq 0$,  the process $(X_1^N(Nt+u), u>0)$ converges in distribution to the stationary
process of the number of jobs of an $M/M/1$ queue with service rate $\lambda$ and arrival
rate $2\mu(\beta-\Psi(t))$.  
\end{theorem}
\noindent
It is easily seen that  the asymptotic expansion
$\Psi(t)\sim \beta-\beta\exp(-2(\beta+\mu t)/\rho)$ holds as $t$ goes to infinity. 
The last part of the theorem states that, ``around'' time $Nt$, the process $X_1^N$ has a local equilibrium. 
\begin{proof}
Equation~\eqref{SDE2} gives that, for $f\in C_c(\N)$ 
\begin{multline}\label{eq1}
f(X^N_1(Nt))-f(X^N_1(0))-M_{f,1}^N(Nt)=
N^2\int_0^t \Omega[Y_N(u)](f)(X(Nu))\,\diff u\\
+\mu N\int_0^t\Delta^-(f)(X_1^N(Nu))X_1^N(Nu)\,\diff u
\end{multline}
with, from Equation~\eqref{SDE1},
\[
Y_N(u)=\frac{F_N}{N}-\frac{X^N_1(Nu)}{N}-\frac{M_0^N(Nu)}{N}-\frac{\mu}{N}\int_0^{Nu} X^N_1(v)\,\diff u,
\]
and $\Delta^-(f)(x)=(f(x-1)-f(x))\ind{x\geq 1}$. 
The bound on the increasing process of the martingale $(M_{f,1}^N(t))$  at the end of Section~\ref{modelsec}, Doob's Inequality  and Lemma~\ref{domlem} show that   the  sequence  of   processes  
\begin{multline*}
\left(\frac{1}{N^2}\left[\rule{0mm}{5mm}f(X^{N}_1(Nt)){-}f(X^{N}_1(0)){-}M_{f,1}^N(Nt)\right.\right. \\
\left.\left.{-}\mu N\int_0^t(\Delta^-(f)(X_1^N(Nu))X_1^N(Nu)\,\diff u\right]\right)
\end{multline*}
converges to $0$   for  the  topology  of  the uniform  norm  on  compact  sets. 

By using Lemma~\ref{domlem}, one gets that
\[
\frac{X^N_1(Nu)}{N}\leq \frac{L_{\beta_0}(N^2u)}{N},
\]
hence the sequence of processes $({X^N_1(N^2u)}/{N})$  converges  in distribution  to $0$. 

The bound on the increasing process and  Proposition~\ref{proppsi} show  that the sequence  of processes $(Y_{N_k}(t))$  converges   in  distribution  to  $(\beta-\Psi(t))$.    One  deduces  from Equation~\eqref{eq1} that the sequence of processes
\[
\left(\int_0^t \Omega[\beta-\Psi(u)](f)(X_1^{N_k}(N_ku))\,\diff u\right)
=\left(\int_0^t \Omega[\beta-\Psi(u)](f)(x)\mu_{N_k}(\diff x,\diff u)\right)
\]
converges to $0$. 

The convergence of the $(\mu_{N_k})$ and Proposition~\ref{proppsi} give therefore that,
almost surely,  the relations
\[
\int_0^t \croc{\pi_u,\Omega[\beta-\Psi(u)](f)}\,\diff u=0 \text{ and } 
\int_0^t \pi_u(\N)\,\diff u = t,
\]
hold for all $t\geq 0$ and all functions $f\in C_c(\N)$. Note that one has used the fact
that $C_c(\N)$ has a countable dense subset for the uniform norm. 

If $\Delta$ is the subset of all real numbers $u\geq 0$ such that  one of the relations 
\[
\begin{cases}
\pi_u(\N)\not=1,\\  
\croc{\pi_u,\Omega[\beta-\Psi(u)](f)}\not=0, \text{ for some } f\in C_c(\N),
\end{cases}
\]
holds, then the Lebesgue measure of $\Delta$ is $0$. Hence if $u\not\in \Delta$, then
$\pi_u(\N)=1$ and $\croc{\pi_u,\Omega[\beta-\Psi(u)](f)}=0$ for all $f\in C_c(\N)$. Since 
$\Omega[\beta-\Psi(u)]$ is the infinitesimal generator of an $M/M/1$ queue with arrival
rate $2\mu(\beta-\Psi(u))$ and service rate $\lambda$, one gets that $\pi_u$ is a
geometric distribution on $\N$ with parameter $2\mu(\beta-\Psi(u))/\lambda$. 

From Proposition~\ref{proppsi} one gets that, for $t\geq 0$, 
\[
\Psi(t)=\mu\int_{[0,t]\setminus\Delta} \croc{\pi_u,I}\,\diff u=\mu\int_0^t \frac{2\mu(\beta-\Psi(u))}{\lambda-2\mu(\beta-\Psi(u))}\,\diff u,
\]
straightforward calculus gives the relation
\[
(\beta-\psi(u))^{\rho/2} e^{\psi(u)}=\beta^{\rho/2} e^{-\mu u}.
\]
It is easily checked that since $2\beta<\rho$, there is a unique $\Psi(u)<\beta$ satisfying
the above equation. The theorem is proved. 
\end{proof}
The theorem gives directly the following corollary on the asymptotic behavior of
$T_N(\delta)$, the first time when a fraction $\delta$ of the files has been lost.  
\begin{corollary}
If $\rho=\lambda/\mu>2\beta$ then, for $N\geq 1$ and $\delta\in(0,1)$, 
\[
T_N(\delta)=\inf\{t\geq 0: X_0^N(t)\geq \delta F_N\},
\]
then, for the convergence in distribution, 
\[
\lim_{N\to+\infty} \frac{T_N(\delta)}{N}= \frac{1}{\mu}\left(-\frac{\rho}{2}\log(1-\delta)-\delta\beta\right).
\]
\end{corollary}

\appendix
\appendix
\section*{Appendix. Generalized Skorokhod Problems}\label{SkoSec}
For the sake of self-containedness, this section presents quickly the more or less classical material necessary to state and prove the convergence results used in this paper.  The general theme concerns the rigorous definition of a solution of a stochastic differential equation constrained to stay in some domain and also the proof of the existence and uniqueness of such a solution. See Skorokhod~\cite{Skorokhod}, Anderson and Orey~\cite{Anderson}, Chaleyat-Maurel and El~Karoui~\cite{Elkaroui} and, in a multi-dimensional context, Harrison and Reiman~\cite{Harrison} and Taylor and Williams~\cite{Taylor} and, in a more general context, Ramanan~\cite{Ramanan}. See Appendix~D of Robert~\cite{Robert} for a brief account.  

We first recall the classical definition of Skorokhod problem in dimension~$K$.  If $(Z(t))$ is some function of the set ${\cal D}(\R_+,\R)$ of c\`adl\`ag functions defined on $\R_+$, the couple of functions $[(X(t)),(R(t))]$ is said to be a solution of the Skorokhod problem associated to $(Z(t))$ and $P$ whenever
\begin{enumerate}
\item $X(t)=Z(t)+  R(t)$, for all $t\geq 0$,
\item $X(t)\geq 0$, for all $t\geq 0$,
\item $t\to R(t)$ is non-decreasing, $R(0)=0$ and
\[
\int_{\R_+} X(t)\,\diff R(t)=0.
\]
\end{enumerate}
 The generalization used in this paper corresponds to the case when $(Z(t))$ is itself a functional of $(X(t))$.
\begin{defi}[Generalized Skorokhod Problem]\ \\
If $G: {\cal D}(\R_+,\R)\to {\cal D}(\R_+,\R)$ is a Borelian function, $((X(t)),(R(t)))$ is a solution of the generalized Skorokhod Problem (GSP) associated to $G$ if $((X(t)),(R(t)))$ is the solution of the Skorokhod Problem associated to $G(X)$ and $P$, in particular, for all $t\geq 0$, 
\[
X(t)=G(X)(t)+ R(t)\text{ and }
\int_{\R_+} X(t)\,\diff R(t)=0,
\]
\end{defi}
The classical Skorokhod problem described above corresponds to the case when the functional
$G$ is constant and equal to $(Z(t))$. If one takes
\[
G(x)(t)=\int_0^t\sigma(x(u))\,\,\diff B(u) +\int_0^t\delta(x(u))\,\diff u,
\]
where  $(B(t))$ is a standard Brownian motion and $\sigma$ and $\delta$ are Lipschitz functions on
$\R$.  The first coordinate $(X(t))$ of a possible solution to the corresponding GSP can
be described as the solution of the SDE  
\[
\diff X(t)=\sigma(X(t))\,\,\diff B(t)+\delta(X(t))\,\diff t
\]
reflected at $0$.\\
\begin{prop}\label{GSPprop}
If $G: {\cal D}(\R_+,\R)\to {\cal D}(\R_+,\R)$ is such that, for any $T>0$, there exists a constant $C_T$
such that, for all  $(x(t))\in {\cal D}(\R_+,\R)$ and $0\leq t\leq T$,
\begin{equation}\label{Lip}
\sup_{0\leq s\leq t} \|G(x)(s)-G(y)(s)\|\leq C_T \int_0^t \|x(u)-y(u)\|\,\diff u,
\end{equation}
then there exists a unique solution to the
generalized Skorokhod problem associated to the functional $G$ and the matrix $P$. 
\end{prop}
\begin{proof}
Define the sequence $(X_N(t))$ by induction $(X^0(t),R^0(t))=0$ and, for $N\geq 1$,  $(X^{N+1},R^{N+1})$ is the
solution of the Skorokhod problem  (SP) associated to $G(X^{N})$, in particular,
\[
X^{N+1}(t)=F\left(X^N\right)(t)+R^{N+1}(t) \text{ and } \int_{\R_+} X^{N+1}(u)\,\diff R^{N+1}(u)=0.
\]
The existence of such a solution is guaranteed as well as the Lipschitz property of the solutions of a classical Skorokhod problem, see Proposition~D.4 of
Robert~\cite{Robert}, this gives the existence of some constant $K_T$ such that, for all
$N\geq 1$ and $0\leq t\leq T$, 
\[
\left\|X^{N+1}-X^{N}\right\|_{\infty,t}\leq K_T \left\|F\left(X^{N}\right)-F\left(X^{N-1}\right)\right\|_{\infty,t},
\]
where $\|h\|_{\infty,T}= \sup\{|h(s)|:0\leq s\leq T\}$. From Relation~\eqref{Lip}, this
implies that
\[
\left\|X^{N+1}-X^{N}\right\|_{\infty,t}\leq \alpha\int_0^{t} \left\|X^{N}-X^{N-1}\right\|_{\infty,u}\,\diff u,
\]
with $\alpha=K_TC_T$. The iteration of the last relation yields the inequality
\[
\left\|X^{N+1}-X^{N}\right\|_{\infty,t}\leq \frac{(\alpha t)^N}{N!} \int_0^{t}
\left\|X^{1}\right\|_{\infty,u}\,\diff u, \quad 0\leq t\leq T.
\]
One concludes that the sequence $(X^N(t))$ is converging uniformly on compact sets and
consequently the same is true for the sequence $(R^N(t))$. Let $(X(t))$ and $(R(t))$ be the
limit of these sequences. By continuity of the SP, the couple $((X(t)), (R(t)))$ is the
solution of the SP associated to $G(X)$, and hence a solution of the GSP associated to $F$. 

Uniqueness. If $(Y(t))$ is another solution of the GSP associated to $F$. In the same way
as before, one gets by induction, for $0\leq t\leq T$,
\[
\left\|X-Y\right\|_{\infty,t}\leq \frac{(\alpha t)^N}{N!} \int_0^{t} \left\|X-Y\right\|_{\infty,u}\,\diff u,
\]
and by letting $N$ go to infinity, one concludes that $X=Y$. The proposition is proved. 
\end{proof}

\providecommand{\bysame}{\leavevmode\hbox to3em{\hrulefill}\thinspace}
\providecommand{\MR}{\relax\ifhmode\unskip\space\fi MR }
\providecommand{\MRhref}[2]{%
  \href{http://www.ams.org/mathscinet-getitem?mr=#1}{#2}
}
\providecommand{\href}[2]{#2}

\end{document}